\documentclass[MSNbibl,number,citesort,dvips]{arxbj}
\usepackage{upgreek}

\aid{0}
\volume{19}
\issue{2}
\pubyear{2013}
\firstpage{599}
\lastpage{609}
\doi{10.3150/11-BEJ404}

\makeatletter
\newcommand{\eqref}[1]{(\ref{#1})}
\newtheorem{theorem}{Theorem}
\newtheorem{lemma}[theorem]{Lemma}
\makeatother

\begin{document}
\begin{frontmatter}

\title{Parisian ruin probability for spectrally negative L\'evy processes}
\runtitle{Parisian ruin probability for spectrally negative L\'evy processes}

\begin{aug}
\author[1]{\fnms{Ronnie} \snm{Loeffen}\thanksref{1}\ead[label=e1]{loeffen@wias-berlin.de}},
\author[2]{\fnms{Irmina} \snm{Czarna}\thanksref{2,e2}\ead[label=e2,mark]{czarna@math.uni.wroc.pl}}
\and
\author[2]{\fnms{Zbigniew} \snm{Palmowski}\corref{}\thanksref{2,e3}\ead[label=e3,mark]{zbigniew.palmowski@gmail.com}}
\runauthor{R. Loeffen, I. Czarna and Z. Palmowski} 
\address[1]{Weierstrass Institute, Mohrenstrasse 39, 10117 Berlin,
Germany. \printead{e1}}
\address[2]{Department of Mathematics, University of Wroc\l aw, pl.
Grunwaldzki 2/4, 50-384 Wroc\l aw, Poland.
\printead{e2,e3}}
\end{aug}

\received{\smonth{2} \syear{2011}}
\revised{\smonth{8} \syear{2011}}

%
\begin{abstract}
In this note we give, for a spectrally negative L\'evy process, a
compact formula for the Parisian ruin probability, which is defined by
the probability that the process exhibits an excursion below zero, with
a length that exceeds a certain fixed period $r$. The formula involves
only the scale function of the spectrally negative L\'evy process and
the distribution of the process at time $r$.
\end{abstract}

%
\begin{keyword}
\kwd{L\'{e}vy process}
\kwd{Parisian ruin}
\kwd{risk process}
\kwd{ruin probability}
\end{keyword}

\end{frontmatter}

\section{Introduction}

Let $X=\{X_t,t\geq0\}$ be a spectrally negative L\'evy process on the
filtered probability space $(\Omega,\mathcal{F},\{\mathcal{F}_t\dvt t\geq0\}
, \mathbb P)$; that is, $X$ is a stochastic process issued from the
origin which has stationary and independent increments and cadlag paths
that have no positive jump discontinuities. To avoid degenerate cases,
we exclude the case where $X$ has monotone paths. As a strong Markov
process, we shall endow $X$ with probabilities $\{\mathbb P_x\dvt x\in
\mathbb R\}$, such that under $\mathbb P_x$, we have $X_0=x$ with
probability one. Further, $\mathbb E_x$ denotes expectation with
respect to $\mathbb P_x$. Recall that $\mathbb P=\mathbb P_0$ and
$\mathbb E=\mathbb E_0$. For background on spectrally negative L\'evy
processes, we refer to Section 8 of~\cite{kypbook}.

In this paper we deal with the quantity $\kappa_r$ with $r>0$, which is
defined by
\[
\kappa_r=\inf\{ t>r\dvt  t-g_t>r \}, \qquad \mbox{where $g_t=\sup\{0\leq
s\leq t\dvt X_s\geq0\}$.}
\]
Hereby we make the convention that $\inf\varnothing=\infty$ and $\sup
\varnothing=0$.
The stopping time $\kappa_r$ is the first time the process has stayed
below zero for a consecutive period of length greater than $r$,
and here we are interested in the probability that such an excursion
occurs. The number $r$ is referred to as the delay.
Such \textit{Parisian} stopping times have been studied by Chesney \textit{et al.}~\cite{chesney} in the context of barrier options in
mathematical finance. Dassios and Wu~\cite{dassioswusemi} introduced $\kappa
_r$ in actuarial risk theory:  the process $X$ under $\mathbb P_x$ with
$x\geq0$ is then used as a model for the surplus process of an
insurance company with initial capital $x$, and the company is said to
be Parisian ruined if an excursion as described above occurs. We
therefore call $\kappa_r$ the Parisian ruin time and the probability of
the event $\{\kappa_r<\infty\}$ under $\mathbb P_x$ the Parisian ruin
probability. Note that in risk theory, the classical term ruin is
referred to as the event that the surplus process reaches a strictly
negative level.

In another paper Dassios and Wu~\cite{dassioswu} gave a formula for the
Parisian ruin probability and, more generally, the Laplace transform of
the distribution of the Parisian ruin time in the case where $X$ is a
Brownian motion plus drift and in the case where $X$ is the classical
compound Poisson risk process with exponentially distributed claims.
Recently, Czarna and Palmowski~\cite{czarnapalmowski} gave a
description of the Parisian ruin probability for a general spectrally
negative L\'evy process; hereby they split the analysis into two cases:
one in which the underlying L\'evy process has paths of bounded
variation and one in which the process has paths of unbounded
variation. In particular, their result allowed them to reproduce the
formulas of Dassios and Wu~\cite{dassioswu} in the two aforementioned
cases, and to tackle also the case where $X$ is the classical compound
Poisson risk process perturbed by Brownian motion and with
exponentially distributed claims.

Another relevant paper is the one of Landriault \textit{et al.} \cite
{landriault}. Here the authors study, for a spectrally negative L\'evy
process of bounded variation, a somewhat different type of Parisian
stopping time, in which, loosely speaking, the deterministic, fixed
delay $r$ is replaced by an independent exponential random variable
with a fixed parameter $p>0$. To be a little bit more precise, each
time the process starts a new excursion below zero, a new independent
exponential random variable with parameter $p$ is considered, and the
stopping time of interest, let us denote it by $\kappa_{\exp(p)}$, is
defined as the first time when the length of the excursion is bigger
than the value of the accompanying exponential random variable.
Although in insurance the stopping time $\kappa_{\exp(p)}$ is arguably
less interesting than $\kappa_r$; working with exponentially
distributed delays allowed the authors to obtain relatively simple
expressions, for example, the Laplace transform of $\kappa_{\exp(p)}$
in terms of the so-called ($q$-)scale functions of $X$. In order to
avoid a misunderstanding, we emphasize that, in the definition of
$\kappa_{\exp(p)}$, by~\cite{landriault}, there is not a single
underlying exponential random variable, but a whole sequence (each
attached to a separate excursion below zero); therefore $\mathbb
P_x(\kappa_{\exp(p)}\in\mathrm{d}z)$ does not equal $\int_0^\infty
p\mathrm{e}^{-pr} \mathbb P_x(\kappa_r\in\mathrm{d}z)$.

As the main result of our paper, we provide an expression for the
Parisian ruin probability, which is considerably simpler than the one
of Czarna and Palmowski~\cite{czarnapalmowski}, and which
simultaneously holds for spectrally negative L\'evy processes of
bounded and unbounded variation.
Before stating this result, we introduce a little extra notation.

The Laplace exponent of $X$ is denoted by $\psi(\theta)$, that is,
\[
\psi(\theta)=\log\mathbb E [ \mathrm{e}^{\theta X_1} ],
\]
which is well defined for $\theta\geq0$. We further introduce the scale
function $W$ of $X$ (cf.~\cite{kypbook}, Section~8), which is the
strictly increasing, continuous function uniquely defined on $[0,\infty
)$ through its Laplace transform which is given by
%
\begin{equation}
\label{scalefunction}
\int_0^\infty\mathrm{e}^{-\theta x}W(x)\,\mathrm{d}x = \frac{1}{\psi
(\theta)},\qquad  \theta>0.
\end{equation}
We extend $W$ to the whole real line by setting $W(x)=0$ for
$x<0$.\vadjust{\goodbreak}

We now give the promised expression for the Parisian ruin probability.
Hereby we assume that~$X$ drifts to infinity (equivalently, $\mathbb
E[X_1]>0$) as otherwise Parisian ruin happens with probability one.
Recall that $r$ stands for a strictly positive real number.
\begin{theorem}\label{mainresult}
Assume $\mathbb E[X_1]>0$. Then for any $x\in\mathbb R$,
%
\begin{equation}
\label{newformula}
\mathbb{P}_x(\kappa_r<\infty) =
1 - \mathbb E[X_1]\frac{\int_0^\infty W(x+z) z\mathbb{P}(X_r\in\mathrm
{d}z) }{ \int_0^\infty z\mathbb{P}(X_r\in\mathrm{d}z)}.
\end{equation}
\end{theorem}

We remark that the methodology of the proof of Theorem~\ref{mainresult}
can also be used to derive the Laplace transform of $\kappa_r$.
However, since the proof would be considerably longer and the resulting
expression much harder to (e.g., numerically) evaluate, we stick to
presenting only the Parisian ruin probability.

In the next section we introduce some additional notation and recall
some facts about spectrally negative L\'evy processes, which are needed
for the proof of Theorem~\ref{mainresult} given in Section~\ref
{section_proof}. We conclude by giving some specific examples for which
the Parisian ruin probability can be expressed in a somewhat more
explicit form.

\section{Preliminaries}

We define for $a\in\mathbb R$,
\[
\tau_a^+=\inf\{t>0\dvt  X_t>a\}, \qquad \tau_a^-=\inf\{t>0\dvt X_t<a\}.
\]
Note that since $0$ is regular for the upper half-line for $X$ (cf.~\cite{kypbook}, Section 8.1), we have, by the strong Markov property,
$\tau_a^+=\inf\{t>0\dvt  X_t\geq a\}$.

If $\mathbb E[X_1]>0$ (equivalently, $\psi'(0)>0$), then
%
\begin{equation}
\label{ruinprobability}
\mathbb P_x(\tau_0^-<\infty)= 1-\mathbb E[X_1]W(x);
\end{equation}
cf.~\cite{kypbook}, equation (8.7). We see in particular that, in this
case, $W(x)$ is bounded by $1/\mathbb E[X_1]$.

The scale function $W$ is absolutely continuous on $(0,\infty)$ (cf.
\cite{kypbook}, Lemma 8.2), and we denote a version of its density by
$W'$. Further, $W(0)>0$ if $X$ has paths of bounded variation, and
$W(0)=0$ if $X$ has paths of unbounded variation; cf.~\cite{kypbook}, Lemma 8.6.
We further recall the following well-known expression for the Laplace
transform of the first passage time above $a$:
%
\begin{equation}
\label{passagabove}
\mathbb{E}_x [ \mathrm{e}^{-\theta\tau_a^+} ] = \mathrm
{e}^{\Phi(\theta)(x-a)}, \qquad  x\leq a, \theta\geq0,
\end{equation}
where $\Phi(\theta)=\sup\{\lambda\geq0\dvt \psi(\lambda)=\theta\}$; cf.
\cite{kypbook}, Section 8.1.
We will also use Kendall's identity (cf.~\cite{bertoin}, Corollary
VII.3), which relates the distribution of a spectrally negative
L\'{e}vy process to the distribution of its upward passage time $\tau_z^+$,
%
\begin{equation}\label{kendallsidentity}
r \mathbb P(\tau_z^+\in\mathrm{d} r)\,\mathrm{d} z = z \mathbb P(X_r\in
\mathrm{d} z)\,\mathrm{d} r.
\end{equation}
Last, we need the following three identities.
\begin{lemma}
For $\theta>0$,
\begin{eqnarray}
\label{laplaceovershoot}
\mathbb E_x \bigl[ \mathbf{1}_{\{\tau_0^-<\infty\}} \mathrm{e}^{\Phi
(\theta)X_{\tau_0^-}} \bigr]
&= & \frac{\theta}{\Phi(\theta)} \int_0^\infty\mathrm{e}^{-\Phi(\theta
)y}W'(x+y)\,\mathrm{d}y, \\
\label{lemmapart1}
\int_0^\infty\mathrm{e}^{-\theta r} \int_{y}^\infty\frac{z}{r}
\mathbb P(X_r\in\mathrm{d}z)\,\mathrm{d}r &= & \frac{1}{\Phi(\theta
)}\mathrm{e}^{-\Phi(\theta)y}, \qquad  y\geq0, \\
\label{niceidentity}
\int_0^\infty W(z) \frac{z}{r}\mathbb{P}(X_r\in\mathrm{d}z)& = & 1.
\end{eqnarray}
\end{lemma}
\begin{pf}
Using the known identity (cf.~\cite{loeffenrenaud}, equation (35))
\[
\label{laplruinovershoot}
\mathbb{E}_x \bigl[ \mathrm{e}^{p X_{\tau_0^-}} \mathbf{1}_{\{\tau
_0^-<\infty\}} \bigr] \\
= \mathrm{e}^{px} -\psi(p)\mathrm{e}^{px}\int_0^x \mathrm
{e}^{-pz}W(z)\,\mathrm{d}z - \frac{\psi(p)}{p}W(x), \qquad  p>0,
\]
by \eqref{scalefunction} and a change of variables and an integration
by parts, we get
\begin{eqnarray*}
\mathbb{E}_x \bigl[ \mathrm{e}^{\Phi(\theta)X_{\tau_0^-}} \mathbf{1}_{\{
\tau_0^-<\infty\}} \bigr]& = &
\mathrm{e}^{\Phi(\theta)x}\biggl ( 1-\theta\int_0^x \mathrm{e}^{-\Phi
(\theta)z}W(z)\,\mathrm{d}z \biggr) - \frac{\theta}{\Phi(\theta)}W(x) \\
&= & \mathrm{e}^{\Phi(\theta)x} \theta\int_x^\infty\mathrm{e}^{-\Phi
(\theta)z}W(z)\,\mathrm{d}z - \frac{\theta}{\Phi(\theta)}W(x) \\
& = & \theta\int_0^\infty\mathrm{e}^{-\Phi(\theta)y }W(x+y)\,\mathrm{d}y
- \frac{\theta}{\Phi(\theta)}W(x) \\
&= & \frac{\theta}{\Phi(\theta)} \int_0^\infty\mathrm{e}^{-\Phi(\theta
)y }W'(x+y)\,\mathrm{d}y.
\end{eqnarray*}
This proves the first identity. For the second, we have, by Kendall's
identity \eqref{kendallsidentity}, Tonelli and \eqref{passagabove},
\begin{eqnarray*}
\int_0^\infty\mathrm{e}^{-\theta r} \int_y^\infty \frac{z}{r}\mathbb
{P}(X_r\in\mathrm{d}z) \,\mathrm{d}r
&= & \int_0^\infty\mathrm{e}^{-\theta r} \int_y^\infty \mathbb{P}(\tau
_z^+\in\mathrm{d}r)\,\mathrm{d}z \\
&= & \int_y^\infty\mathrm{e}^{-\Phi(\theta)z}\,\mathrm{d}z = \frac{1}{\Phi
(\theta)}\mathrm{e}^{-\Phi(\theta)y}.
\end{eqnarray*}
In order to prove the last identity, we use again Kendall's identity
\eqref{kendallsidentity}, Tonelli and \eqref{passagabove} combined with
\eqref{scalefunction} to get
\begin{eqnarray*}
\int_0^\infty\mathrm{e}^{-\theta r} \int_0^\infty W(z) \frac
{z}{r}\mathbb{P}(X_r\in\mathrm{d}z) \,\mathrm{d}r
&= & \int_0^\infty\mathrm{e}^{-\theta r} \int_0^\infty W(z) \mathbb
P(\tau_z^+\in\mathrm{d}r)\,\mathrm{d}z \\
&= & \int_0^\infty\mathrm{e}^{-\Phi(\theta)z} W(z)\,\mathrm{d}z
= \frac{1}{\theta}.
\end{eqnarray*}
Hence \eqref{niceidentity} follows by Laplace inversion (note that the
left-hand side of \eqref{niceidentity} is continuous in $r$ because a
L\'evy process is continuous in probability, and $W$ is bounded and
continuous). Note that \eqref{niceidentity} can also directly be
deduced from~\cite{bertoin}, Corollary VII.16.
\end{pf}

\section{\texorpdfstring{Proof of Theorem \protect\ref{mainresult}}{Proof of Theorem 1}}\label{section_proof}

In the case where $X$ has paths of unbounded variation, we will use a
limiting argument in the proof. For this reason we introduce for
$\varepsilon\geq0$ the stopping time $\kappa_r^\varepsilon$, which is defined by
\[
\kappa_r^\varepsilon=\inf\{ t>r\dvt  t-g^\varepsilon_t>r, X_{t-r}<0 \},
\qquad\mbox{where $g^\varepsilon_t=\sup\{0\leq s\leq t\dvt X_s\geq\varepsilon\}$.}
\]
The stopping time $\kappa_r^\varepsilon$ is the first time that an
excursion starting when $X$ gets below zero, ending before $X$ gets
back up to $\varepsilon$ and having length greater than $r$, has occurred.
Note that $\kappa_r^0=\kappa_r$. We have for $x<0$, by the strong
Markov property and the absence of upward jumps,
%
\begin{eqnarray}
\label{ruinbelow0}
\mathbb P_x(\kappa_r^\varepsilon<\infty)& = & \mathbb P_x (\tau_\varepsilon
^+>r) + \mathbb{E}_x [ \mathbb P_x ( \tau_\varepsilon^+\leq
r,\kappa_r^\varepsilon<\infty|\mathcal{F}_{\tau_\varepsilon^+} )
] \nonumber\\
&= & \mathbb P_x (\tau_\varepsilon^+>r) + \mathbb P_x(\tau_\varepsilon^+\leq
r)\mathbb P_\varepsilon(\kappa_r^\varepsilon<\infty) \\
&= & 1 - \mathbb P_x(\tau_\varepsilon^+\leq r) \bigl( 1-\mathbb P_\varepsilon
(\kappa_r^\varepsilon<\infty) \bigr).\nonumber
\end{eqnarray}
Using the above we have, for $x\geq0$, again by the strong Markov property,
%
\begin{eqnarray}
\label{ruinabove0}
\mathbb P_x(\kappa_r^\varepsilon<\infty) &= & \mathbb E_x [ \mathbb
P_x ( \kappa_r^\varepsilon<\infty|\mathcal{F}_{\tau_0^-} )
] \nonumber\\
&= & \mathbb E_x \bigl[ \mathbf{1}_{\{\tau_0^-<\infty\}} \mathbb
P_{X_{\tau_0^-}} ( \kappa_r^\varepsilon<\infty) \bigr] \\
&= & \mathbb P_x(\tau_0^-<\infty) -\bigl( 1-\mathbb P_\varepsilon(\kappa
_r^\varepsilon<\infty) \bigr) \mathbb E_x \bigl[ \mathbf{1}_{\{\tau
_0^-<\infty\}} \mathbb P_{X_{\tau_0^-}} (\tau_\varepsilon^+\leq r)
\bigr].\nonumber
\end{eqnarray}

We proceed by finding an expression in terms of the scale function $W$
and the law of $X_r$ for the expectation in the right-hand side of
\eqref{ruinabove0}.
By spatial homogeneity, we can write, for $x\geq0$,
\begin{eqnarray*}
\mathbb E_x \bigl[ \mathbf{1}_{\{\tau_0^-<\infty\}} \mathbb P_{X_{\tau
_0^-}} (\tau_\varepsilon^+\leq r) \bigr]& = &
\int_{[0, \infty)} \mathbb E_x \bigl[ \mathbf{1}_{\{\tau_0^-<\infty,
-X_{\tau_0^-}\in\mathrm{d}z \}} \mathbb P_{-z} (\tau_\varepsilon^+\leq r)
\bigr] \\
&= & \int_{[0,\infty)} \mathbb E_x \bigl[ \mathbf{1}_{\{\tau_0^-<\infty
, -X_{\tau_0^-}\in\mathrm{d}z \}} \mathbb P (\tau_{\varepsilon+z}^+\leq r)
\bigr].
\end{eqnarray*}
Since we have for $\theta,z>0$ by an integration by parts and \eqref
{passagabove},
\[
\int_0^\infty\mathrm{e}^{-\theta r} \mathbb P (\tau_{\varepsilon+z}^+\leq
r)\,\mathrm{d}r =\frac{1}{\theta}\mathbb E [ \mathrm{e}^{-\theta\tau
_{\varepsilon+z}^+} ] = \frac{1}{\theta}\mathrm{e}^{-\Phi(\theta
)(z+\varepsilon)},
\]
we deduce in combination with Tonelli, \eqref{laplaceovershoot} and
\eqref{lemmapart1}, that for $\theta>0$ and $x\geq0$,
%
\begin{eqnarray}
\label{laplace_difficultterm}
&&\int_0^\infty\mathrm{e}^{-\theta r} \mathbb E_x\bigl [ \mathbf{1}_{\{
\tau_0^-<\infty\}} \mathbb P_{X_{\tau_0^-}} (\tau_\varepsilon^+\leq r)
\bigr] \,\mathrm{d}r \nonumber\\
&&\quad=
\frac{1}{\theta}\mathrm{e}^{-\Phi(\theta)\varepsilon} \int_0^\infty
\mathbb E_x \bigl[ \mathbf{1}_{\{\tau_0^-<\infty\}} \mathrm{e}^{\Phi
(\theta)X_{\tau_0^-}} \bigr]
\nonumber
\\[-8pt]
\\[-8pt]
\nonumber
&&\quad=  \int_0^\infty\frac{ \mathrm{e}^{-\Phi(\theta)(y+\varepsilon)} }{\Phi
(\theta)} W'(x+y)\,\mathrm{d}y \\
&&\quad=  \int_0^\infty W'(x+y) \int_0^\infty\mathrm{e}^{-\theta r} \int
_{y+\varepsilon}^\infty\frac{z}{r} \mathbb P(X_r\in\mathrm{d}z)\,\mathrm
{d}r \,\mathrm{d}y.\nonumber
\end{eqnarray}
Hence by Tonelli and Laplace inversion (noting that both sides of the
equation below are right-continuous in $r$), for $x\geq0$,
%
\begin{eqnarray}
\label{keycomputation}
\mathbb E_x \bigl[ \mathbf{1}_{\{\tau_0^-<\infty\}} \mathbb P_{X_{\tau
_0^-}} (\tau_\varepsilon^+\leq r) \bigr]
&= & \int_\varepsilon^\infty\int_0^{z-\varepsilon} W'(x+y)\,\mathrm{d}y \frac
{z}{r} \mathbb P(X_r\in\mathrm{d}z)
\nonumber
\\[-8pt]
\\[-8pt]
\nonumber
&= & \int_\varepsilon^\infty[ W(x+z-\varepsilon)-W(x) ] \frac{z}{r} \mathbb
P(X_r\in\mathrm{d}z).
\end{eqnarray}


The next step is to prove \eqref{newformula} for $x=0$. For this we
split the analysis into two cases, $W(0)>0$ and $W(0)=0$. First we
consider the case $W(0)>0$ (or equivalently $X$ has paths of bounded
variation). Then by \eqref{ruinprobability}, $\mathbb P (\tau_0^-<\infty
)<1$ and thus $\mathbb E [ \mathbf{1}_{\{\tau_0^-<\infty\}}
\mathbb P_{X_{\tau_0^-}} (\tau_0^+\leq r) ]<1$. Using \eqref
{ruinabove0} and~\eqref{keycomputation} with $x=\varepsilon=0$ and~\eqref
{ruinprobability} and \eqref{niceidentity}, we get
\begin{eqnarray*}
\mathbb P(\kappa_r<\infty) &= & \frac{\mathbb P(\tau_0^-<\infty) - \int
_0^\infty[ W(z)-W(0) ] ({z}/{r}) \mathbb P(X_r\in\mathrm{d}z) }{ 1 -
\int_0^\infty[ W(z)-W(0) ] ({z}/{r}) \mathbb P(X_r\in\mathrm{d}z) }
\\
&= & \frac{-W(0)\mathbb E[X_1] + W(0)\int_0^\infty ({z}/{r}) \mathbb
P(X_r\in\mathrm{d}z) }{ W(0) \int_0^\infty({z}/{r}) \mathbb P(X_r\in
\mathrm{d}z) } \\
&= & \frac{\int_0^\infty ({z}/{r}) \mathbb P(X_r\in\mathrm{d}z)
-\mathbb E[X_1] }{ \int_0^\infty({z}/{r}) \mathbb P(X_r\in\mathrm
{d}z) }.
\end{eqnarray*}

Now we deal with the second case, that is, we assume $W(0)=0$. Let
$\varepsilon>0$. Then $\mathbb P_\varepsilon(\tau_0^-<\infty)<1$ (cf. \eqref
{ruinprobability}) and consequently using \eqref{ruinabove0} and \eqref
{keycomputation} with $x=\varepsilon$ and \eqref{ruinprobability}, we deduce
%
\begin{eqnarray}
\label{beforelimit}
&&\mathbb P_\varepsilon(\kappa_r^\varepsilon<\infty) \nonumber\\
&&\quad=  \frac{\mathbb
P_\varepsilon(\tau_0^-<\infty) - \int_\varepsilon^\infty[ W(z)-W(\varepsilon) ]
({z}/{r}) \mathbb P(X_r\in\mathrm{d}z) }{ 1 - \int_\varepsilon^\infty[
W(z)-W(\varepsilon) ] ({z}/{r}) \mathbb P(X_r\in\mathrm{d}z) }
\\
&&\quad=  \frac{ 1/W(\varepsilon) -\mathbb E[X_1] +\int_\varepsilon^\infty
(z/r)\mathbb P(X_r\in\mathrm{d}z) - \int_{\varepsilon}^\infty
({W(z)}/{W(\varepsilon)})  (z/r)\mathbb P(X_r\in\mathrm{d}z) } {1/W(\varepsilon
) +\int_\varepsilon^\infty (z/r)\mathbb P(X_r\in\mathrm{d}z) - \int
_{\varepsilon}^\infty({W(z)}/{W(\varepsilon)})  (z/r)\mathbb P(X_r\in
\mathrm{d}z) }.\nonumber
\end{eqnarray}
We now want to compute the limit as $\varepsilon\downarrow0$ of both
sides of \eqref{beforelimit}. To this end, recalling that $0<\mathbb
E[X_1]<\infty$, we have by \eqref{niceidentity}, an integration by
parts and l'H\^opital,
%
\begin{eqnarray}
\label{limitrhs}
&&\lim_{\varepsilon\downarrow0} \frac{1-\int_\varepsilon^\infty W(z)
(z/r)\mathbb P(X_r\in\mathrm{d}z) }{W(\varepsilon)}\nonumber\\
&&\quad =  \lim_{\varepsilon
\downarrow0} \int_0^\varepsilon\frac{W(z)}{W(\varepsilon)}  (z/r)\mathbb
P(X_r\in\mathrm{d}z)
\nonumber
\\[-8pt]
\\[-8pt]
\nonumber
&&\quad=  \lim_{\varepsilon\downarrow0} \biggl( \int_0^\varepsilon\frac yr\mathbb
P(X_r\in\mathrm{d}y) - \int_0^\varepsilon\frac{W'(z)}{W(\varepsilon)} \int
_0^z \frac yr\mathbb P(X_r\in\mathrm{d}y) \,\mathrm{d}z \biggr) \\
&&\quad=  0 - \lim_{\varepsilon\downarrow0} \frac{W'(\varepsilon) \int_0^\varepsilon
 (y/r)\mathbb P(X_r\in\mathrm{d}y) }{W'(\varepsilon)} =0.\nonumber
\end{eqnarray}
For the limit as $\varepsilon\downarrow0$ of the left-hand side of \eqref
{beforelimit}, we introduce, for $\varepsilon>0$, the stopping time
\[
\widetilde{\kappa}^\varepsilon_r=\inf\{ t>r\dvt  t-g_t>r, X_{t-r}<-\varepsilon \}
, \qquad \mbox{where $g_t=\sup\{0\leq s\leq t\dvt X_s\geq0\}$.}
\]
We easily see that for $0<\varepsilon'<\varepsilon$, $\{\widetilde{\kappa
}_r^\varepsilon<\infty\} \subset\{\widetilde{\kappa}_r^{\varepsilon'}<\infty
\}$ and $\bigcup_{\varepsilon>0}\{\widetilde{\kappa}_r^{\varepsilon}<\infty\}=\{
\kappa_r<\infty\}$.
Hence, by spatial homogeneity,
\[
\lim_{\varepsilon\downarrow0} \mathbb P_\varepsilon(\kappa_r^\varepsilon<\infty)
= \lim_{\varepsilon\downarrow0} \mathbb P(\widetilde\kappa_r^\varepsilon
<\infty) = \mathbb P(\kappa_r<\infty),
\]
and, combined with \eqref{beforelimit} and \eqref{limitrhs}, this leads to
%
\begin{equation}
\label{parisianruin0}
\mathbb P(\kappa_r<\infty) = \frac{\int_0^\infty ({z}/{r}) \mathbb
P(X_r\in\mathrm{d}z) -\mathbb E[X_1] }{ \int_0^\infty({z}/{r})
\mathbb P(X_r\in\mathrm{d}z) }.
\end{equation}

Hence, recalling \eqref{niceidentity}, we have shown in both cases that
\eqref{newformula} holds for $x=0$. Now plugging \eqref
{ruinprobability} and the above expression for $\mathbb P(\kappa
_r<\infty)$ into \eqref{ruinabove0} and using \eqref{keycomputation}
with $\varepsilon=0$, we arrive at \eqref{newformula} for any $x\geq0$.

We now prove that \eqref{newformula} also holds for $x<0$.
For $x<0$, setting $\varepsilon=0$ in \eqref{ruinbelow0} and applying
\eqref{newformula} (with $x=0$) in combination with \eqref
{niceidentity}, we get
%
\[
\mathbb{P}_x (\kappa_r<\infty) = 1 - \mathbb E[X_1]\frac{r\mathbb
P_x(\tau_{0}^+\leq r) }{ \int_0^\infty z\mathbb{P}(X_r\in\mathrm{d}z)}.
\]
The proof will be completed once we show that for $x<0$
\[
\mathbb P_x(\tau_{0}^+\leq r)=\frac{1}{r}\int_0^\infty W(x+z) z\mathbb
{P}(X_r\in\mathrm{d}z).
\]
The above identity follows by showing that the
Laplace transforms in $r$ of both sides are equal. To this end, note
that by an integration by parts and \eqref{passagabove}, the Laplace
transform of the left-hand side is equal to
\[
\int_0^\infty\mathrm{e}^{-\theta r} \mathbb P_x(\tau_{0}^+\leq
r)\,\mathrm{d}r = \frac{1}{\theta}\mathbb E [ \mathrm e^{-\theta\tau
_{-x}^+} ] =\frac{1}{\theta} \mathrm{e}^{\Phi(\theta)x},
\]
and by Tonelli, Kendall's identity \eqref{kendallsidentity}, \eqref
{passagabove} and \eqref{scalefunction}, the Laplace transform of the
left-hand side equals
\begin{eqnarray*}
\int_0^\infty\mathrm e^{-\theta r}\frac{1}{r}\int_0^\infty W(x+z)
z\mathbb{P}(X_r\in\mathrm{d}z)\,\mathrm{d}r
&= & \int_{0}^\infty W(x+z)\int_0^\infty\mathrm e^{-\theta r} \mathbb
{P}(\tau_z^+\in\mathrm{d} r)\,\mathrm{d}z \\
&= & \int_{0}^\infty W(x+z) \mathrm e^{-\Phi(\theta)z}\, \mathrm dz \\
&= & \frac{1}{\psi(\Phi(\theta))} \mathrm e^{\Phi(\theta)x} \\
&= & \frac{1}{\theta} \mathrm{e}^{\Phi(\theta)x}.
\end{eqnarray*}
Note that in the third equality we used that $x<0$.

\section{Examples}

If one wants to evaluate formula \eqref{newformula}, one needs to know
the scale function $W$ and the distribution of $X_r$. There are plenty
of examples of spectrally negative L\'evy processes for which the
distribution at a fixed time has a closed-form expression and, thanks
to recent developments (for the latest we refer to Section 3 of Chazal
\textit{et al.}~\cite{chazalkyprianoupatie}), there are also a lot of
examples for which the scale function is known explicitly.
Unfortunately, it seems that there are only a few examples for which
both of them are known in closed-form. Below we give three examples for
which this is the case and for which we give the corresponding
expression for the Parisian ruin probability. In the other case, one
can resort to numerically inverting the Laplace transform (which recall
is explicitly given in terms of $\psi$) of one or both of the two quantities.

\subsection{Brownian motion}

Let $X_t=\mu t+\sigma B_t$, with $\mu,\sigma>0$ and $\{B_t,t\geq0\}$ a
Brownian motion. Then $\psi(\theta)=\mu\theta+\frac12\sigma^2\theta^2$,
and we easily deduce from \eqref{scalefunction} that
\[
W(x)=\frac{1}{\mu} \bigl( 1-\mathrm{e}^{-2({\mu}/{\sigma^2}) x}
\bigr), \qquad  x\geq0 . 
\]
Hence by \eqref{newformula} for $x\geq0$,
\[
\mathbb{P}_x (\tau_r<\infty)
=  \mathrm{e}^{-2({\mu}/{\sigma^2})x} \frac{\int_0^\infty \mathrm
{e}^{-2({\mu}/{\sigma^2}) z} z\mathbb{P}(X_r\in\mathrm{d}z) }{ \int
_0^\infty z\mathbb{P}(X_r\in\mathrm{d}z)}.
\]
Setting $x=0$ in above formula, comparing with \eqref{parisianruin0}
and realizing that $\mathbb E[X_r]=r\mathbb E[X_1]=\mu r$, we see that
\[
\int_0^\infty \mathrm{e}^{-2({\mu}/{\sigma^2}) z} z\mathbb{P}(X_r\in
\mathrm{d}z) = \int_0^\infty z\mathbb{P}(X_r\in\mathrm{d}z)-\mu r.
\]
Noting that $X_r$ has a normal distribution with mean $\mu r$ and
variance $\sigma^2 r$ and making the change of variables $y=\frac{z-\mu
r}{\sigma\sqrt{r}}$, we get
\begin{eqnarray*}
\int_0^\infty z \mathbb{P}(X_r\in\mathrm{d}z) &= & \frac{1}{\sqrt{2\uppi
\sigma^2 r}}\int_0^\infty z\mathrm{e}^{-{(z-\mu r)^2}/(2\sigma^2
r)}\,\mathrm{d}z \\
&= & \frac{\sigma\sqrt{r}}{\sqrt{2\uppi}} \int_{-{\mu\sqrt r}/{\sigma
}}^\infty y\mathrm{e}^{-{y^2}/{2} }\,\mathrm{d}y + \frac{\mu r}{\sqrt
{2\uppi}} \int_{-{\mu\sqrt r}/{\sigma}}^\infty\mathrm{e}^{-
{y^2}/{2} }\,\mathrm{d}y \\
&= & \frac{\sigma\sqrt{r}}{\sqrt{2\uppi}}\mathrm{e}^{-{\mu^2r}/({2\sigma
^2})} + \mu r \mathcal N \biggl( \frac{\mu\sqrt r}{\sigma} \biggr),
\end{eqnarray*}
where $\mathcal{N}$ is the cumulative distribution function of a
standard normal random variable.
Putting everything together results in
\begin{eqnarray*}
\mathbb{P}_x (\tau_r<\infty) &=& \mathrm{e}^{-2({\mu}/{\sigma^2})x}
\biggl( \frac{\sigma\sqrt{r}}{\sqrt{2\uppi}}\mathrm{e}^{-{\mu
^2r}/(2\sigma^2)} - \mu r\mathcal N \biggl( -\frac{\mu\sqrt r}{\sigma}
\biggr) \biggr)\\
&&{}\Big/\biggl(\frac{\sigma\sqrt{r}}{\sqrt{2\uppi}}\mathrm{e}^{-{\mu
^2r}/(2\sigma^2)} + \mu r\mathcal N \biggl( \frac{\mu\sqrt r}{\sigma}
\biggr)\biggr),  \qquad  x\geq0,
\end{eqnarray*}
which agrees with the formula found by Dassios and Wu~\cite{dassioswu}.

\subsection{Classical risk process with exponential claims}

Let $X_t=ct-\sum_{i=1}^{N_t}C_i$, where $\{N_t,t\geq0\}$ is a Poisson
process with rate $\eta$ and $C_1,C_2,\ldots$ are i.i.d. exponentially
distributed random variables with parameter $\alpha$ independent of
$N_t$. Its Laplace exponent is given by $\psi(\theta)=c\theta-\eta+\eta
\frac{\alpha}{\theta+\alpha}$. Assume that $c>\eta/\alpha$, so that
$\mathbb E[X_1]=\psi'(0)=c-\eta/\alpha>0$. From \eqref{scalefunction}
it easily follows that
\[
W(x)= \frac{1}{c-\eta/\alpha} \biggl( 1- \frac{\eta}{c\alpha}\mathrm
{e}^{({\eta}/{c}-\alpha)x} \biggr), \qquad  x\geq0.
\]
Recalling that a sum of i.i.d. exponential random variables equals a
gamma random variable and utilizing the independence between $C_i$ and
$N_t$, we get
\begin{eqnarray*}
\mathbb{P} \Biggl( \sum_{i=1}^{N_r}C_i\in\mathrm{d}y \Biggr)& = & \sum
_{k=0}^\infty\mathbb P \Biggl( \sum_{i=0}^k C_i\in\mathrm{d}y \Biggr)
\mathbb P(N_r=k) \\
&= & \mathrm{e}^{-\eta r}\delta_0(\mathrm{d}y) + \sum_{k=1}^\infty\alpha
^k\frac{y^{k-1}\mathrm{e}^{-\alpha y}}{(k-1)!}\frac{(\eta
r)^k}{k!}\mathrm{e}^{-\eta r} \,\mathrm{d}y \\
&= & \mathrm{e}^{-\eta r} \Biggl( \delta_0(\mathrm{d}y) + \mathrm
{e}^{-\alpha y}\sum_{m=0}^\infty\frac{ (\alpha\eta r)^{m+1}}{m!(m+1)!}
y^{m} \,\mathrm{d}y \Biggr),
\end{eqnarray*}
where $\delta_0(\mathrm{d}y)$ is the Dirac mass at $0$. It follows that
\begin{eqnarray*}
\int_0^\infty z\mathbb{P}(X_r\in\mathrm{d}z) &= & \int_0^{cr} z\mathrm
{e}^{-\eta r} \Biggl( \delta_0(cr-\mathrm{d}z) + \mathrm{e}^{-\alpha
(cr-z)}\sum_{m=0}^\infty\frac{ (\alpha\eta r)^{m+1}}{m!(m+1)!}
(cr-z)^{m} \,\mathrm{d}z \Biggr) \\
&= & \mathrm{e}^{-\eta r} \Biggl( cr + \sum_{m=0}^\infty\frac{ (\alpha
\eta r)^{m+1}}{m!(m+1)!} \int_0^{cr} \mathrm{e}^{-\alpha y} (cr-y)y^{m}
\,\mathrm{d}y \Biggr) \\
&= & \mathrm{e}^{-\eta r} \Biggl( cr + \sum_{m=0}^\infty\frac{ (\eta
r)^{m+1}}{m!(m+1)!} \biggl[ cr\Gamma(m+1,cr\alpha) -\frac1\alpha\Gamma
(m+2,cr\alpha) \biggr] \Biggr),
\end{eqnarray*}
where $\Gamma(a,x)=\int_0^x \mathrm{e}^{-t}t^{a-1}\,\mathrm{d}t$ is the
incomplete gamma function. Putting everything together and using the
same trick as in the previous example leads, for $x\geq0$, to
\begin{eqnarray*}
&&\mathbb{P}_x(\kappa_r<\infty)\\
&&\quad=  \mathrm{e}^{({\eta}/{c}-\alpha)x}
\frac{\int_0^\infty z\mathbb{P}(X_r\in\mathrm{d}z)-(c-\eta/\alpha)
r}{\int_0^\infty z\mathbb{P}(X_r\in\mathrm{d}z)} \\
&&\quad=  \mathrm{e}^{({\eta}/{c}-\alpha)x} \Biggl( 1- {\mathrm
{e}^{\eta r}(c-\eta/\alpha) }\\
&&\hspace*{62pt}\qquad{}\big/ \Biggl(c + \sum_{m=0}^\infty\frac{ (\eta
r)^{m+1}}{m!(m+1)!} \biggl[ c \Gamma(m+1,cr\alpha) -\frac1{\alpha r}
\Gamma(m+2,cr\alpha) \biggr] \Biggr) \Biggr).
\end{eqnarray*}
Although it does not seem easy to show directly that the above
expression is equal to the one found by Dassios and Wu \cite
{dassioswu}, one can check numerically that the two expressions match.

\subsection{Stable process with index $3/2$}

Let $X_t=ct+Z_t$, where $c>0$ and $\{Z_t,t\geq0\}$ is a spectrally
negative $\alpha$-stable process with $\alpha=3/2$.
The Laplace exponent of $X$ is given by $\psi(\theta)=c\theta+ \theta
^{3/2}$. One can, in a straightforward way, check via \eqref
{scalefunction} that
\[
W(x)=\frac1c\bigl[1-E_{1/2}\bigl(-c\sqrt{x}\bigr)\bigr], \qquad  x\geq0,
\]
where $E_{1/2}(z)$ is the Mittag--Leffler function of order 1/2, that
is, $E_{1/2}(z)=\break\sum_{k=0}^\infty\frac{z^k}{\Gamma((1/2)k+1)}$.
It follows that, for $x\geq0$,
\[
\mathbb{P}_x(\kappa_r<\infty)= \frac{\int_{0}^\infty E_{1/2}(-c\sqrt
{x+z}) z \mathbb P(Z_r\in\mathrm{d}z-cr)}{\int_{0}^\infty z \mathbb
P(Z_r\in\mathrm{d}z-cr)}.
\]
Using the scaling property of stable processes (see, e.g.,~\cite{bertoin}, Section VIII.1) and the expression for the density of the
underlying stable distribution taken from Schneider~\cite{schneider}, equations (3.4) and~(3.5), we get the following expression
for the distribution of $Z_r$:
\[
\mathbb P(Z_r\in\mathrm{d}y)= \mathbb P(r^{2/3}Z_1\in\mathrm{d}y) =
\cases{
\displaystyle\sqrt{\frac3{\uppi}} r^{2/3} y^{-1} \mathrm{e}^{-u/2}W_{1/2,1/6}(u)\,\mathrm
{d}y, & \quad $y>0,$ \vspace*{2pt}\cr
\displaystyle-\frac{1}{2\sqrt{3\uppi}}r^{2/3} y^{-1} \mathrm
{e}^{u/2}W_{-1/2,1/6}(u)\,\mathrm{d}y, &\quad $ y<0,$}
\]
where $u=\frac{4}{27}r^{9/2}|y|^3$ and $W_{\kappa,\mu}(z)$ is
Whittaker's W-function; see, for example,~\cite{lebedev}, Section~9.13.

\section*{Acknowledgement}
The authors are indebted to an anonymous referee for pointing out that
\eqref{newformula} holds for $x<0$, as well.


\printhistory


\begin{thebibliography}{12}

\bibitem{bertoin}
\begin{bbook}[mr]
\bauthor{\bsnm{Bertoin},~\bfnm{Jean}\binits{J.}}
(\byear{1996}).
\btitle{L\'evy Processes}.
\bseries{Cambridge Tracts in Mathematics}
\bvolume{121}.
\baddress{Cambridge}: \bpublisher{Cambridge Univ. Press}.
\bid{mr={1406564}}
\bptok{imsref}%
\end{bbook}
\endbibitem

\bibitem{chazalkyprianoupatie}
\begin{bmisc}[auto:STB|2011/12/15|13:36:40]
\bauthor{\bsnm{Chazal},~\bfnm{M.}\binits{M.}},
\bauthor{\bsnm{Kyprianou},~\bfnm{A.E.}\binits{A.E.}} \AND
\bauthor{\bsnm{Patie},~\bfnm{P.}\binits{P.}}
(\byear{2010}).
\bhowpublished{A transformation for L\'evy
processes with one-sided jumps and applications. Preprint. Available at
\href{http://arXiv.org/abs/1010.3819v1}{arXiv:}
\href{http://arXiv.org/abs/1010.3819v1}{1010.3819v1 [math.PR]}.}
\bptok{imsref}%
\end{bmisc}
\endbibitem

\bibitem{chesney}
\begin{barticle}[mr]
\bauthor{\bsnm{Chesney},~\bfnm{Marc}\binits{M.}},
  \bauthor{\bsnm{Jeanblanc-Picqu{\'e}},~\bfnm{Monique}\binits{M.}} \AND
  \bauthor{\bsnm{Yor},~\bfnm{Marc}\binits{M.}}
(\byear{1997}).
\btitle{Brownian excursions and {P}arisian barrier options}.
\bjournal{Adv. in Appl. Probab.}
\bvolume{29}
\bpages{165--184}.
\bid{doi={10.2307/1427865}, issn={0001-8678}, mr={1432935}}
\bptok{imsref}%
\end{barticle}
\endbibitem

\bibitem{czarnapalmowski}
\begin{bmisc}[auto:STB|2011/12/15|13:36:40]
\bauthor{\bsnm{Czarna},~\bfnm{I.}\binits{I.}}
\AND
\bauthor{\bsnm{Palmowski},~\bfnm{Z.}\binits{Z.}}
(\byear{2011}).
\bhowpublished{Ruin probability with Parisian delay for a
spectrally negative L\'evy risk process. \textit{J. Appl. Probab.} \textbf{48} 984--1002.}
\bptok{imsref}%
\end{bmisc}
\endbibitem

\bibitem{dassioswusemi}
\begin{bmisc}[auto:STB|2011/12/15|13:36:40]
\bauthor{\bsnm{Dassios},~\bfnm{A.}\binits{A.}}
\AND
\bauthor{\bsnm{Wu},~\bfnm{S.}\binits{S.}}
\bhowpublished{Semi-Markov model for excursions and occupation time
  of Markov processes. Working paper, LSE, London. Available at
\href{http://stats.lse.ac.uk/angelos/}{http://stats.lse.ac.uk/angelos/}}.
\bptok{imsref}%
\end{bmisc}
\endbibitem

\bibitem{dassioswu}
\begin{bmisc}[auto:STB|2011/12/15|13:36:40]
\bauthor{\bsnm{Dassios},~\bfnm{A.}\binits{A.}}
\AND
\bauthor{\bsnm{Wu},~\bfnm{S.}\binits{S.}}
\bhowpublished{Parisian ruin with exponential claims. Working
paper, LSE, London. Available at \href{http://stats.lse.ac.uk/angelos/}{http://stats.lse.ac.uk/angelos/}}.
\bptok{imsref}%
\end{bmisc}
\endbibitem

\bibitem{kypbook}
\begin{bbook}[mr]
\bauthor{\bsnm{Kyprianou},~\bfnm{Andreas~E.}\binits{A.E.}}
(\byear{2006}).
\btitle{Introductory Lectures on Fluctuations of {L}\'evy Processes with
  Applications}.
\bseries{Universitext}.
\baddress{Berlin}: \bpublisher{Springer}.
\bid{mr={2250061}}
\bptok{imsref}%
\end{bbook}
\endbibitem

\bibitem{landriault}
\begin{bmisc}[auto:STB|2011/12/15|13:36:40]
\bauthor{\bsnm{Landriault},~\bfnm{D.}\binits{D.}}
\bauthor{\bsnm{Renaud},~\bfnm{J.-F.}\binits{J.-F.}} \AND
\bauthor{\bsnm{Zhou},~\bfnm{X.}\binits{X.}}
(\byear{2010}).
\bhowpublished{Insurance risk models with
  Parisian implementation delays. Preprint}.
\bptok{imsref}%
\end{bmisc}
\endbibitem

\bibitem{lebedev}
\begin{bbook}[mr]
\bauthor{\bsnm{Lebedev},~\bfnm{N.~N.}\binits{N.N.}}
(\byear{1965}).
\btitle{Special Functions and Their Applications},
\bedition{revised English} ed.
\baddress{Englewood Cliffs, NJ}: \bpublisher{Prentice-Hall Inc.}
\bnote{Translated and edited by Richard A. Silverman}.
\bid{mr={0174795}}
\bptok{imsref}%
\end{bbook}
\endbibitem

\bibitem{loeffenrenaud}
\begin{barticle}[mr]
\bauthor{\bsnm{Loeffen},~\bfnm{Ronnie~L.}\binits{R.L.}} \AND
  \bauthor{\bsnm{Renaud},~\bfnm{Jean-Fran{\c{c}}ois}\binits{J.F.}}
(\byear{2010}).
\btitle{De {F}inetti's optimal dividends problem with an affine penalty
  function at ruin}.
\bjournal{Insurance Math. Econom.}
\bvolume{46}
\bpages{98--108}.
\bid{doi={10.1016/j.insmatheco.2009.09.006}, issn={0167-6687}, mr={2586160}}
\bptok{imsref}%
\end{barticle}
\endbibitem

\bibitem{schneider}
\begin{bincollection}[mr]
\bauthor{\bsnm{Schneider},~\bfnm{W.~R.}\binits{W.R.}}
(\byear{1986}).
\btitle{Stable distributions: {F}ox functions representation and
  generalization}.
In \bbooktitle{Stochastic Processes in Classical and Quantum Systems ({A}scona,
  1985)}.
\bseries{Lecture Notes in Physics}
\bvolume{262}
\bpages{497--511}.
\baddress{Berlin}: \bpublisher{Springer}.
\bid{doi={10.1007/3540171665_92}, mr={0870201}}
\bptok{imsref}%
\end{bincollection}
\endbibitem

\end{thebibliography}
\end{document}